\pdfoutput=1
\RequirePackage{ifpdf}
\ifpdf 
\documentclass[pdftex]{sigma}
\else
\documentclass{sigma}
\fi

\numberwithin{equation}{section}

\newtheorem{Theorem}{Theorem}[section]
\newtheorem{Corollary}[Theorem]{Corollary}
\newtheorem{Lemma}[Theorem]{Lemma}
\newtheorem{Proposition}[Theorem]{Proposition}
{ \theoremstyle{definition}
\newtheorem{Definition}[Theorem]{Definition}
}

\def\t{\theta}

\begin{document}

\allowdisplaybreaks

\newcommand{\arXivNumber}{1605.09775}

\renewcommand{\PaperNumber}{103}

\FirstPageHeading

\ShortArticleName{Strictly Positive Def\/inite Kernels on a Product of Spheres II}

\ArticleName{Strictly Positive Def\/inite Kernels\\ on a Product of Spheres II}

\Author{Jean C.~GUELLA, Valdir A.~MENEGATTO and Ana P.~PERON}

\AuthorNameForHeading{J.C.~Guella, V.A.~Menegatto and A.P.~Peron}

\Address{Instituto de Ci\^encias Matem\'aticas e de Computa\c{c}\~ao, Universidade de S\~ao Paulo,\\
Caixa Postal 668, 13560-970, S\~ao Carlos - SP, Brazil}
\Email{\href{mailto:jeanguella@gmail.com}{jeanguella@gmail.com}, \href{menegatt@icmc.usp.br}{menegatt@icmc.usp.br}, \href{apperon@icmc.usp.br}{apperon@icmc.usp.br}}

\ArticleDates{Received June 01, 2016, in f\/inal form October 24, 2016; Published online October 28, 2016}

\Abstract{We present, among other things, a necessary and suf\/f\/icient condition for the strict positive def\/initeness of an isotropic and positive def\/inite kernel on the cartesian \mbox{product} of a circle and a higher dimensional sphere. The result complements similar results previously obtained for strict positive def\/initeness on a product of circles [\textit{Positivity}, to appear, arXiv:1505.01169] and on a product of high dimensional spheres [\textit{J.~Math. Anal. Appl.} \textbf{435} (2016), 286--301, arXiv:1505.03695].}


\Classification{33C50; 33C55; 42A16; 42A82; 42C10; 43A35}

\section{Introduction}

The theory of positive def\/inite and strictly positive def\/inite kernels on manifolds and groups can not be separated from the seminal paper of I.J.~Schoenberg~\cite{schoen} published in the f\/irst half of the past century. The major contribution in that paper refers to kernels of the form $(x,y) \in S^m \times S^m \to f(x \cdot y)$, in which $\cdot$ is the usual inner product of $\mathbb{R}^{m+1}$ and the ``generating function'' $f$ is real and continuous in $[-1,1]$. The kernel is positive def\/inite if, and only if, the generating function $f$ has a series representation in the form
\begin{gather*}
f(t)=\sum_{k=0}^{\infty}a_{k}^mP_{k}^m(t),\qquad t \in [-1,1],
\end{gather*}
where all the coef\/f\/icients $a_k^m$ are nonnegative, $P_k^m$ denotes the usual Gegenbauer polynomial of degree $k$ attached to the rational number $(m-1)/2$ and $\sum_{k}a_k^m P_k^m(1)<\infty$. The normalization for the Gegenbauer polynomials used here is
\begin{gather*}
P_n^m(1)=\left(\begin{matrix} n+m-2 \\ n \end{matrix} \right),\qquad n=0,1,\ldots.
\end{gather*}
Recall that the positive def\/initeness of a general real kernel $(x,y) \in X^2 \to K(x,y)$ on a nonempty set $X$, demands that $K(x,y)=K(y,x)$, $x,y\in X$, and the inequality
\begin{gather*} \sum_{\mu,\nu=1}^nc_\mu c_\nu K(x_\mu , x_\nu)\geq 0,
\end{gather*}
whenever $n$ is a positive integer, $x_1, x_2, \ldots, x_n$ are $n$ distinct points on~$X$ and $c_1, c_2, \ldots, c_n$ are real scalars. For complex kernels we do not need the symmetry assumption and we have to use complex scalars instead. Isotropy on $S^m$ refers to the fact that the variables~$x$ and~$y$ of~$S^m$ are tied to each other via the inner product of~$\mathbb{R}^{m+1}$.

Some f\/ifty years later, the very same positive def\/inite functions were found useful for solving scattered data interpolation problems on spheres. But that demanded {\em strictly} positive def\/inite functions, and thus a characterization of these functions was needed at the start. We recall that the strict positive def\/initeness of a general positive def\/inite kernel as in the above def\/inition requires that the previous inequalities be strict whenever at least one~$c_\mu$ is nonzero. In other words, the {\em interpolation matrices} $[K(x_\mu , x_\nu)]_{\mu,\nu=1}^n$ of~$K$ at each set $\{x_1, x_2, \ldots, x_n\}$ need to be positive def\/inite. The strict positive def\/initeness on spheres was an issue for some time until Schoenberg's result was complemented by a result of Debao Chen et al.\ in 2003~\cite{chen} and by Menegatto et al.~\cite{menegatto}. A real, continuous, isotropic and positive def\/inite kernel $(x,y) \in S^m \times S^m \to f(x \cdot y)$ is strictly positive def\/inite if, and only if, the following additional condition holds for the coef\/f\/icients in Schoenberg's series representation for the generating function~$f$:
\begin{itemize}\itemsep=0pt
\item[--] $(m=1)$ the set $\{k \in \mathbb{Z}\colon a_{|k|}^1>0\}$ intersects each full arithmetic progression in $\mathbb{Z}$;
\item[--] $(m\geq 2)$ the set $\{k \colon a_k^m >0\}$ contains inf\/initely many even and inf\/initely many odd integers.
\end{itemize}

As a matter of fact, positive def\/inite and strictly positive def\/inite functions and kernels on spheres play a fundamental role in several other applications. For instance, two recent papers authored by Beatson and Zu Castell~\cite{beat,beat1} provide new families of strictly positive def\/inite functions on spheres via the so-called half-step operators, a spherical analogue of Matheron's mont\'ee and descente operators on~$\mathbb{R}^{m+1}$. Additional applications are mentioned in~\cite{berg,gneiting}.

The spheres belong to a much larger class of metric spaces, that is, they are compact two-point homogeneous spaces. Positive def\/initeness on these spaces in the same sense we explained above was investigated by Gangolli~\cite{gangolli} while strict positive def\/initeness was completely characterized in~\cite{barbosa}.

In 2011, the paper~\cite{emonds} considered strictly positive def\/inite functions on compact abelian groups taking into account a paper on strict positive def\/initeness on~$S^1$ previously written by X.~Sun~\cite{xingping}. Among other things, the paper included abstract characterizations for strict positive def\/initeness on a torsion group and on a product of a f\/inite group and a torus.

In the past two years, the attention shifted all the way to positive def\/initeness on a product of spaces, the main motivation coming from problems involving random f\/ields on spaces across time. The f\/irst important reference we would like to mention along this line is~\cite{berg}, where the authors investigated positive def\/inite kernels on a product of the form $G\times S^m$, in which $G$ is an arbitrary locally compact group, keeping both the group structure of~$G$ and the isotropy of~$S^m$ in the setting. Let us denote by $e$ the neutral element of $G$, $*$ the operation of the group $G$ and by $u^{-1}$ the inverse of $u \in G$ with respect to $*$. The main contribution in~\cite{berg} states that a~continuous kernel of the form $((u,x), (v,y)) \in (G \times S^m)^2 \to f(u^{-1}*v, x\cdot y)$ is positive def\/inite if, and only if, the function~$f$ has a representation in the form
\begin{gather*}f(u,t)=\sum_{n=0}^{\infty}f_n^m(u) P_n^m(t), \qquad (u,t) \in G \times [-1,1],\end{gather*}
in which $\{f_n^m\}$ is a sequence of continuous functions on $G$ for which $\sum f_n^m(e)P_n^m(1)<\infty$, with uniform convergence of the series for $(u,t) \in G\times [-1,1]$. As a matter of fact, the functions~$f_n^m$ are positive def\/inite on~$G$ in the sense that the kernel $(u,v) \in G^2 \to f_n^m(u^{-1}*v)$ is positive def\/inite as previously def\/ined. The paper~\cite{berg} did not considered any strict positive def\/initeness issues. It is worth to mention that~\cite{berg} may be linked to~\cite{gneit} where the reader can f\/ind a possible reason for considering positive def\/initeness on a product of a locally compact abelian group with a classical space.

Simultaneously, positive def\/initeness and strict positive def\/initeness on a product of spheres was investigated in \cite{jean0,jean,jean1}. A~characterization for the isotropic and positive def\/inite kernels on $S^m \times S^M$ was deduced in~\cite{jean} and that agreed with the characterization mentioned in \cite{berg} and also with \cite[Chapter~4]{shapiro}. By the way, a real, continuous and isotropic kernel $((x,z), (y,w)) \in (S^m \times S^M)^2 \to f(x \cdot y, z \cdot w)$ is positive def\/inite if, and only if, the function $f$ has a double series representation in the form
\begin{gather}\label{requ}
f(t,s) = \sum_{k,l=0}^{\infty}a_{k,l}^{m,M}P_{k}^{m}(t)P_{l}^{M}(s), \qquad t,s \in [-1,1],
\end{gather}
in which $a_{k,l}^{m,M}\geq 0$, $k,l\in \mathbb{Z}_+$ and $\sum\limits_{k,l=0}^{\infty}a_{k,l}^{m,M}P_{k}^{m}(1)P_{l}^{M}(1)<\infty$. As before, we will call $f$ the generating function of the kernel.

One of the main theorems in \cite{jean0} reveals that, in the case in which $m,M \geq 2$, a positive def\/inite kernel as above is strictly positive def\/inite if, and only if, the set $\{(k,l)\colon a_{k,l}^{m,M}>0\}$ contains sequences from each one of the sets $2\mathbb{Z}_+ \times 2\mathbb{Z}_+$, $2\mathbb{Z}_+ \times (2\mathbb{Z}_++1)$, $(2\mathbb{Z}_++1) \times 2\mathbb{Z}_+$, and $(2\mathbb{Z}_++1) \times (2\mathbb{Z}_++1)$, all of them having both component sequences unbounded. The very same paper contains some other intriguing results related to strict positive def\/initeness, including a~notion of strict positive def\/initeness that holds in product spaces only. In the case $m=M=1$, the condition becomes this one~\cite{jean1}: the set $\{(k,l)\colon a_{|k|,|l|}^{1,1}>0\}$ intersects all the translations of each subgroup of $\mathbb{Z}^2$ having the form $\{(pa, qb)\colon q,p\in \mathbb{Z}\}$, $a,b>0$. Even with the completion of these papers, it became clear very soon that a similar characterization for the remaining case, that is, strict positive def\/initeness on~$S^1 \times S^m$, $m\geq 2$, would demand much more work, perhaps a mix of the techniques used in~\cite{jean0,jean1}.

This is the point where we explain what the contributions in this paper are. In the next section, we present two abstract results that describe how to obtain continuous, isotropic and strictly positive kernels on $S^1 \times S^m$ via the characterization for strict positive def\/initeness of continuous, isotropic and positive def\/inite kernels on either~$S^1$ or~$S^m$ separately. In Section~\ref{section3}, we present necessary and suf\/f\/icient conditions in order that a real, continuous, isotropic and positive def\/inite kernel on $S^1 \times S^m$, $m\geq 2$, be strictly positive def\/inite, thus f\/illing in the missing gap in the previous papers on the subject. In Section~\ref{section4}, we indicate how to obtain a~similar characterization after we replace the sphere $S^m$ with an arbitrary compact two-point homogeneous space.

\section{Abstract suf\/f\/icient conditions}\label{section2}

In this section, we describe two abstract suf\/f\/icient conditions for strict positive def\/initeness on $S^1 \times S^m$, both derived from strict positive def\/initeness on single spheres. The results show that transferring strict positive def\/initeness from the factors $S^1$ and $S^m$ to strict positive def\/initeness of the product $S^1 \times S^m$ is not so obvious as it seems.

Here and in the next sections, {\em we will assume all the generating functions of the kernels are real-valued continuous functions and that the dimension $m$ in $S^m$ is at least $2$}. But the reader is advised that the results hold true for complex kernels after some obvious modif\/ications.

The results to be presented here will depend upon a technical lemma that provides an alternative formulation for the strict positive def\/initeness of a continuous, isotropic and positive def\/inite kernel on $S^1 \times S^m$, to be described below. Let $(x_1,w_1),(x_2,w_2), \ldots,(x_n,w_n)$ be distinct points on $S^1 \times S^{m}$ and represent the components in $S^1$ in polar form:
\begin{gather*}
x_\mu=(\cos \t_\mu, \sin \t_\mu), \qquad \t_\mu \in [0,2\pi),\qquad \mu=1,2,\ldots,n.
\end{gather*}
Write $A$ to denote the interpolation matrix of $((x,z), (y,w)) \in (S^1 \times S^m)^2 \to f(x \cdot y, z \cdot w)$ at $\{(x_1,z_1),(x_2,z_2),\ldots, (x_n,z_n)\}$ and consider the quadratic form $c^t A c$, where $c^t$ indicates the transpose of $c$. If the kernel is positive def\/inite, then the addition theorem for spherical harmonics~\cite[p.~18]{muller} and the representation~(\ref{requ}) imply that
\begin{gather*}c^{t}Ac = \sum_{k,l=0}^\infty C(k,l,m) a_{k,l}^{1,m}\sum_{j=1}^{d(l,m)}\left|\sum_{\mu=1}^n c_\mu
e^{i k\t_\mu} Y_{l,j}^m(z_\mu)\right|^2,
\end{gather*}
in which $c=(c_1,c_2, \ldots,c_n)$, $\{Y_{l,j}^m\colon j=1,2,\ldots,d(l,m)\}$ is a basis for the space of all spherical harmonics of degree~$l$ in~$m+1$ variables and~$C(k,l,m)$ is a positive constant that depends upon~$k$,~$l$ and~$m$. The deduction of the equality above requires the formulas $P_0^1\equiv 1$ and
\begin{gather*}P_k^1(\cos \t)=\frac{2}{k} \cos k\t,\qquad t \in [0,2\pi),\qquad k=1,2,\ldots.\end{gather*}
The equality $c^tAc=0$ is equivalent to
\begin{gather*}
\sum_{\mu=1}^n c_\mu e^{i k\t_\mu} Y_{l,j}^m(z_\mu)=0, \qquad j=1,2,\ldots,d(l,m), \qquad (k,l)\in \big\{(k,l)\colon a_{k,l}^{1,m}>0\big\}.
\end{gather*}
If this last piece of information holds, we can invoke the addition theorem once again in order to see that
\begin{gather*}
\sum_{\mu=1}^n c_\mu e^{ik\t_\mu} P_l^m(z_\mu\cdot z)=0, \qquad (k,l)\in \big\{(k,l)\colon a_{k,l}^{1,m}>0\big\}, \qquad z\in S^m.
\end{gather*}
Finally, if the condition above holds, we can multiply the equality in it by $e^{-i\t}$ and split the equation once again via the addition theorem
to obtain
\begin{gather*}
\sum_{j=1}^{d(l,m)} \left[\sum_{\mu=1}^n c_\mu e^{i k(\t_\mu-\t)} Y_{l,j}^m(z_\mu)\right]Y_{l,j}^m(z)=0, \qquad (k,l)\in \big\{(k,l)\colon a_{k,l}^{1,m}>0\big\}, \qquad z \in S^m.
\end{gather*}
Since $\{Y_{l,j}^m \colon j=1,2,\ldots,d(l,m)\}$ is linearly independent, we reach
\begin{gather*}
\sum_{\mu=1}^n c_\mu Y_{l,j}^m(z_\mu)e^{i k(\t_\mu-\t)}=0,\qquad j=1,2,\ldots,d(l,m),\qquad (k,l)\in \big\{(k,l)\colon a_{k,l}^{1,m}>0\big\},
\end{gather*}
for $\t\in[0,2\pi)$. Multiplying the real part of the equality in the previous formula by $\cos k\t$ and integrating in $[0,2\pi]$ with respect to~$\t$, we obtain
\begin{gather*}
\sum_{\mu=1}^n c_\mu Y_{l,j}^m(z_\mu)\cos k\t_\mu = 0,\qquad j=1,2,\ldots,d(l,m),\qquad (k,l)\in \big\{(k,l)\colon a_{k,l}^{1,m}>0\big\}.
\end{gather*}
Similarly, it is easily seen that
\begin{gather*}
\sum_{\mu=1}^n c_\mu Y_{l,j}^m(z_\mu)\sin k\t_\mu=0,\qquad j=1,2,\ldots,d(l,m),\qquad (k,l)\in \big\{(k,l)\colon a_{k,l}^{1,m}>0\big\}.
\end{gather*}

The above computations justify the following lemma.

\begin{Lemma} \label{t-alternativa-spd}
Let $f$ be the generating function of an isotropic and positive definite kernel on~\mbox{$S^1 \times S^m$} and consider its series representation~\eqref{requ}. The following statements are equivalent:
\begin{enumerate}\itemsep=0pt
\item[$(i)$] the kernel $((x,z), (y,w)) \in (S^1 \times S^m)^2 \to f(x \cdot y, z \cdot w)$ is strictly positive definite;
\item[$(ii)$] if $n$ is a positive integer and $(x_1,z_1),(x_2,z_2), \ldots, (x_n,z_n)$ are distinct points on $S^1\times S^{m}$, then the only solution $c=(c_1,c_2, \ldots, c_n)$ of the system
\begin{gather*}
\sum_{\mu=1}^n c_\mu e^{ik\t_\mu} P_l^m(z_\mu\cdot z)=0,\qquad (k,l)\in \big\{(k,l)\colon a_{k,l}^{1,m}>0\big\}, \qquad z\in S^m,
\end{gather*}
is $c=0$.
\end{enumerate}
\end{Lemma}

A particular case of the lemma is pertinent (see Theorem 2 in \cite{chen}).

\begin{Lemma} \label{spdm}
Let $g$ be the generating function of an isotropic and positive definite kernel on~$S^m$ and consider its series representation according to Schoenberg. The following statements are equivalent:
\begin{enumerate}\itemsep=0pt
\item[$(i)$] the kernel $(z,w) \in (S^m)^2 \to g(z \cdot w)$ is strictly positive definite;
\item[$(ii)$] if $n$ is a positive integer and $z_1,z_2, \ldots, z_n$ are distinct points on $S^{m}$, then the only solution $c=(c_1,c_2, \ldots, c_n)$ of the system
\begin{gather*}
\sum_{\mu=1}^n c_\mu P_k^m(z_\mu\cdot z)=0,\qquad k\in \big\{k\colon a_{k}^{m}>0\big\}, \qquad z\in S^m,
\end{gather*}
is $c=0$.
\end{enumerate}
\end{Lemma}

Another consequence is this one (see Theorem 5.1 in \cite{ron}).

\begin{Lemma} \label{spd1}
Let $g$ be the generating function of an isotropic and positive definite kernel on~$S^1$ and consider its series representation according to Schoenberg. The following statements are equivalent:
\begin{enumerate}\itemsep=0pt
\item[$(i)$] the kernel $(x, y) \in (S^1)^2 \to g(x \cdot y)$ is strictly positive definite;
\item[$(ii)$] if $n$ is a positive integer and $x_1,x_2, \ldots, x_n$ are distinct points on $S^{1}$ given in polar form $x_\mu=(\cos \t_\mu, \sin \t_\mu)$, $\mu=1,2,\ldots, n$, then the only solution $c=(c_1,c_2, \ldots, c_n)$ of the system
\begin{gather*}
\sum_{\mu=1}^n c_\mu e^{ik\t_\mu}=0,\qquad k\in \big\{k\colon a_{k}^{1}>0\big\},
\end{gather*}
is $c=0$.
\end{enumerate}
\end{Lemma}

If $f$ generates an isotropic positive def\/inite kernel on $S^1 \times S^m$, we will adopt the following notation attached to its double series representation~(\ref{requ}):
\begin{gather*}J_f:=\big\{(k,l)\colon a_{k,l}^{1,m}>0\big\}.\end{gather*}
If $I$ is a subset of $\mathbb{Z}_+$, we will write $I \in {\rm SPD}(S^m)$ to indicate that there exists a continuous and positive def\/inite kernel $(z,w) \in (S^m)^2 \to g(z \cdot w)$ for which the set $\{l\colon a_l^m>0\}$ attached to the series representation for~$g$ in Schoenberg's result is precisely~$I$. This def\/inition is well posed once strict positive def\/initeness does not depend upon the actual values of the numbers $a_l^m$ in the series representation for the generating function but only on the set $\{l\colon a_l^m>0\}$ itself.

The f\/irst important contribution in this section is as follows.

\begin{Theorem}\label{ana1}
Let $f$ be the generating function of an isotropic and positive definite kernel on~\mbox{$S^1 \times S^m$} and consider its series representation~\eqref{requ}. If
\begin{gather*}
\big\{k \colon \{l \colon (k,l)\in J_f\}\in {\rm SPD}\big(S^m\big)\big\}\in {\rm SPD}\big(S^1\big),
\end{gather*}
then the kernel $((x,z), (y,w)) \in (S^1 \times S^m)^2 \to f(x \cdot y, z \cdot w)$ is strictly positive definite on $S^1\times S^m$.
\end{Theorem}

\begin{proof} We will show that, under the assumption \begin{gather*}
\big\{k \colon \{l \colon (k,l)\in J_f\}\in {\rm SPD}\big(S^m\big)\big\}\in {\rm SPD}\big(S^1\big),\end{gather*} the alternative condition in Lemma~\ref{t-alternativa-spd} holds. In particular, the notation employed in that lemma will be adopted here. Let $(x_1,z_1),(x_2,z_2), \ldots,$ $(x_n,z_n)$ be distinct points in $S^1\times S^{m}$ and suppose that
\begin{gather*}
\sum_{\mu=1}^n c_\mu e^{ik\t_\mu} P_l^m(z_\mu\cdot z)=0, \qquad (k,l)\in J_f, \qquad z\in S^m.
\end{gather*}
Let $M$ be a maximal subset of $\{1,2,\ldots,n\}$ that indexes the distinct elements among the $z_j$. Writing $M_j:=\{\mu\colon z_\mu=z_j\}$, $j \in M$, the previous equality becomes
\begin{gather*}
\sum_{j\in M} \left(\sum_{\mu\in M_j} c_\mu e^{ik\t_\mu} \right) P_l^m(z_j\cdot z)=0, \qquad (k,l)\in J_f, \qquad z\in S^m.
\end{gather*}In particular,
\begin{gather*}
\sum_{j\in M} \left(\sum_{\mu\in M_j} c_\mu e^{ik\t_\mu} \right) P_l^m(z_j\cdot z)=0, \qquad l\in \{l\colon (k,l)\in J_f\}, \qquad z\in S^m,
\end{gather*}
whenever $k\in\{k\colon \{l\colon (k,l)\in J_f\} \in {\rm SPD}(S^m)\}$. Since the $z_j$ in the expression above are all distinct, Lemma~\ref{spdm} yields that
\begin{gather*}
\sum_{\mu\in M_j} c_\mu e^{ik\t_\mu}=0, \qquad k \in \big\{k\colon \{l\colon (k,l)\in J_f\} \in {\rm SPD}\big(S^m\big)\big\},
\end{gather*}
for every $j\in M$. An application of Lemma~\ref{spd1} for each $j$ plus the help of our original assumption leads to~$c_\mu=0$, $\mu\in M_j$, $j\in M$. But this corresponds to $c=0$.
\end{proof}

In a similar manner, but with slightly easier arguments, the following cousin theorem can be proved.

\begin{Theorem}
Let $f$ be the generating function of an isotropic and positive definite kernel on~\mbox{$S^1 \times S^m$} and consider its series representation~\eqref{requ}. If
\begin{gather*}\big\{l\colon \{k\colon (k,l)\in J_f\} \in {\rm SPD}\big(S^1\big)\big\} \in {\rm SPD}\big(S^m\big),\end{gather*}
then the kernel $((x,z), (y,w)) \in (S^1 \times S^m)^2 \to f(x \cdot y, z \cdot w)$ is strictly positive definite on~$S^1\times S^m$.
\end{Theorem}

We close this section presenting realizations for the previous theorems. They follow from the characterizations for strict positive def\/initeness on singles spheres described in the Introduction.

\begin{Corollary}
Let $f$ be the generating function for an isotropic and positive definite kernel on~\mbox{$S^1 \times S^m$} and consider its series representation~\eqref{requ}. Either condition below is sufficient for the kernel $((x,z), (y,w)) \in (S^1 \times S^m)^2 \to f(x \cdot y, z \cdot w)$ to be strictly positive definite:
\begin{enumerate}\itemsep=0pt
\item[$(i)$] $J_f$ contains sequences $\{(k_\mu, 2l_{\mu\nu})\colon \mu,\nu \in \mathbb{Z}_+\}$ and $\{(k'_\mu, 2l'_{\mu\nu}+1)\colon \mu,\nu \in \mathbb{Z}_+\}$ so that
\begin{gather*}\{\pm k_\mu\colon \mu\in \mathbb{Z}_+\} \cap (n\mathbb{Z}+j)\neq \varnothing, \qquad j=0,1,\ldots,n-1, \qquad n\geq 1,\\
\{\pm k'_\mu\colon \mu\in \mathbb{Z}_+\} \cap (n\mathbb{Z}+j)\neq \varnothing, \qquad j=0,1,\ldots,n-1, \qquad n\geq 1,\end{gather*}
and
\begin{gather*}\lim_{\nu \to \infty} l_{\mu\nu}=\lim_{\nu \to \infty} l'_{\mu\nu}=\infty, \qquad \mu\in \mathbb{Z}_+.\end{gather*}
\item[$(ii)$] $J_f$ contains sequences $\{(k_{\mu\nu}, 2l_{\mu})\colon \mu,\nu \in \mathbb{Z}_+\}$ and $\{(k'_{\mu\nu}, 2l'_{\mu}+1)\colon \mu,\nu \in \mathbb{Z}_+\}$ so that
\begin{gather*}\{\pm k_{\mu\nu}\colon \nu\in \mathbb{Z}_+\} \cap (n\mathbb{Z}+j)\neq \varnothing, \qquad j=0,1,\ldots,n-1,\qquad n\geq 1, \qquad \mu \in \mathbb{Z}_+,\\
\{\pm k'_{\mu\nu}\colon \nu\in \mathbb{Z}_+\} \cap (n\mathbb{Z}+j)\neq \varnothing, \qquad j=0,1,\ldots,n-1, \qquad n\geq 1, \qquad \mu \in \mathbb{Z}_+,\end{gather*}
and
\begin{gather*}\lim_{\mu \to \infty} l_{\mu}=\lim_{\mu \to \infty} l'_{\mu}=\infty. \end{gather*}
\end{enumerate}
\end{Corollary}

Finally, we would like to point that the previous theorems can be reproduced in other settings, with or without the presence of isotropy (for example $S^m \times S^m$ and the product of~$S^m$ and a~torus). Details will not be included here.

\section[Characterizations for strict positive def\/initeness on $S^1 \times S^m$]{Characterizations for strict positive def\/initeness on $\boldsymbol{S^1 \times S^m}$}\label{section3}

In order to obtain a characterization for the strict positive def\/initeness of an isotropic positive def\/inite kernel on~$S^1 \times S^m$, we need to look at the concept of strict positive def\/initeness in an enhanced form. We begin this section explaining what we mean by that and introducing the additional concepts needed.

It is an obvious matter to see that we can write the generating function $f$ of a positive def\/inite kernel $((x,z), (y,w)) \in (S^1 \times S^m)^2 \to f(x \cdot y, z \cdot w)$ in the form
\begin{gather} \label{alter}
f(t,s)=\sum_{l=0}^{\infty}f_l(t)P_{l}^{m}(s), \qquad t,s \in [-1,1],
\end{gather}
in which
\begin{gather*}f_l(t):=\sum_{k=0}^\infty a_{k,l}^{1,m}P_{k}^1(t), \qquad t \in [-1,1],\qquad l=0,1,\ldots.\end{gather*}
Since $P_l^m(1)\geq 1$, $m\geq 2$, $l=0,1,\ldots$, the series $\sum\limits_{k=0}^\infty a_{k,l}^{1,m}P_{k}^1(1)<\infty$ converges. In particular, each $f_l$ is the generating function of a continuous, isotropic and positive def\/inite kernel on $S^1$.

In the statement of the next lemma, we will employ the following additional notation related to another one we have previously introduced:
\begin{gather*}J_{f}^k:=\{l\colon (k,l) \in J_f \}.\end{gather*}
In particular,
\begin{gather*}\cup_{k} J_f^k=\{l\colon f_l\neq 0\}.\end{gather*}
Among other things, the lemma suggests how a characterization for the strict positive def\/ini\-te\-ness of an isotropic and positive def\/inite kernel on $S^1 \times S^m$ should look like.

\begin{Lemma} \label{prelim} Let $f$ be the generating function of an isotropic and positive definite kernel on~\mbox{$S^1 \times S^m$} and consider the alternative series representation~\eqref{alter} for~$f$. If $p$ is a positive integer, $x_1, x_2, \ldots, x_p$ are distinct points on~$S^1$ and~$c$ is a real vector in $\mathbb{R}^p$, then the continuous function $g$ given by
\begin{gather*}g(s)=\sum_{l \in \cup_k J_f^k} \big\{c^t [f_l(x_i \cdot x_j) ]_{i,j=1}^p c\big\} P_l^m(s), \qquad s \in [-1,1],\end{gather*}
generates an isotropic and positive definite kernel on~$S^m$. If~$c$ is nonzero and the kernel $((x,z), (y,w)) \in (S^1 \times S^m)^2 \to f(x \cdot y, z \cdot w)$ is strictly positive definite, then $(z,w)\in (S^m)^2 \to g(z\cdot w)$ is strictly positive definite as well.
\end{Lemma}
\begin{proof} Write $c=(c_1,c_2, \ldots, c_p)$. If $z_1, z_2, \ldots, z_q$ are distinct points on $S^m$ and $d_1, d_2, \ldots, d_q$ are real numbers, then
\begin{gather*}\sum_{\mu,\nu=1}^q d_\mu d_\nu g(z_\mu \cdot z_\nu)=\sum_{\mu,\nu=1}^q \sum_{i,j=1}^p (d_\mu c_i) (d_\nu c_j) f(x_i \cdot x_j,z_\mu \cdot z_\nu).\end{gather*}
But, the last expression above corresponds to a quadratic form involving $f$ and the $pq$ distinct points $(x_i, z_\mu)$, $i=1,2,\ldots,p$, $\mu=1,2,\ldots,q$, of $S^1 \times S^m$. In particular,
\begin{gather*}\sum_{\mu,\nu=1}^q d_\mu d_\nu g(z_\mu \cdot z_\nu)\geq 0.\end{gather*}
If the real numbers $d_\mu$ are not all zero and $c\neq 0$, then at least one of the scalars $d_\mu c_i$ is likewise nonzero. Further, if $((x,z), (y,w)) \in (S^1 \times S^m)^2 \to f(x \cdot y, z \cdot w)$ is strictly positive def\/inite, then the quadratic form above is, in fact, positive. In particular, $(z,w) \in (S^m)^2 \to g(z\cdot w)$ is strictly positive def\/inite.
\end{proof}

Another useful technical result is as follows.

\begin{Lemma}\label{basic}
Let $f$ be the generating function of an isotropic and positive definite kernel on~\mbox{$S^1 \times S^m$} and consider the alternative series representation~\eqref{alter} for $f$. If $(x_1,z_1),(x_2,z_2),$ $ \ldots, (x_n,z_n)$ are distinct points in $S^1 \times S^m$ and $c_1, c_2, \ldots, c_n$ are real scalars, then the following assertions are equivalent:
\begin{enumerate}\itemsep=0pt
\item[$(i)$] $\sum\limits_{i,j=1}^n c_i c_j f(x_i \cdot x_j,z_i \cdot z_j)=0$;
\item[$(ii)$] $\sum\limits_{i,j=1}^n c_i c_j f_l(x_i \cdot x_j) P_l^m(z_i \cdot z_j)=0$, $l\in \cup_{k}J_f^k$.
\end{enumerate}
\end{Lemma}
\begin{proof}
One direction is immediate while the other follows simply from the observation that $(t,s) \in [-1,1]^2 \to f_l(t)P_l^m(s)$ is the generating function of a positive def\/inite kernel \linebreak on~\mbox{$S^1 \times S^m$}.
\end{proof}

Next, we formalize the def\/inition of enhancement we use in this section.

\begin{Definition} Let $p$ and $q$ be positive integers, $\{x_1,x_2, \ldots, x_p\} \subset S^1$ and $\{z_1,z_2, \ldots, z_q\}$ an antipodal free subset of $S^m$, that is, a set containing no pairs of antipodal points. An {\em enhanced subset} of $S^1 \times S^m$ generated by them is the set
\begin{gather*}\big\{(x_1,z_1), (x_2,z_1), \ldots, (x_p,z_1), (x_1,z_2), (x_2,z_2), \ldots, (x_p,z_2),\ldots, \\
\qquad{} (x_1,z_q),(x_2,z_q),\ldots, (x_p,z_q), (x_1,-z_1), (x_2,-z_1), \ldots, (x_p,-z_1), \\
\qquad{} (x_1,-z_2), (x_2,-z_2), \ldots,(x_p,-z_2),\ldots, (x_1,-z_q), (x_2,-z_q), \ldots, (x_p,-z_q)\big\}.\end{gather*}
\end{Definition}

The positive numbers $p$ and $q$ in the previous def\/inition may have no connection at all. The order in which the elements in an enhanced subset of $S^1 \times S^m$ are displayed is not relevant, but the writing of the upcoming results will take into account the order adopted above and inherited from those in the subsets of $S^1$ and $S^m$ involved. An enhanced set as above contains $2pq$ distinct points.

The following lemma is concerned with the existence of enhanced sets.

\begin{Lemma} \label{key1Lemma} If $B'=\{(x_1',z_1'),(x_2',z_2'), \ldots, (x_n',z_n')\}$ is a subset of $S^1 \times S^m$, then there exists an enhanced subset $B$ of $S^1 \times S^m$ that contains~$B'$.
\end{Lemma}
\begin{proof}
It suf\/f\/ices to consider the enhanced subset of $S^1 \times S^m$ generated by the set $\{x_1,x_2$, $\ldots, x_p\}$ that encompasses the distinct elements among the $x_i'$ and an antipodal free subset $\{z_1,z_2, \ldots, z_q\}$ of~$S^m$ satisfying $z_i' \in \{z_1,z_2, \ldots, z_q\}$, $i=1,2,\ldots,n$.
\end{proof}

If $f$ is the generating function of an isotropic and positive def\/inite kernel on $S^1 \times S^m$ and~$B$ is an enhanced set as previously described, we will write $\mathcal{E}(f,B)$ to denote the interpolation matrix of $((x,z), (y,w)) \in (S^1 \times S^m)^2 \to f(x \cdot y, z \cdot w)$ at $B$, keeping the order for the points of~$B$. If $A$ is a subset of $S^1 \times S^m$ and~$B$ is an enhanced subset of $S^1 \times S^m$ containing $A$, then the interpolation matrix of $((x,z), (y,w)) \in (S^1 \times S^m)^2 \to f(x \cdot y, z \cdot w)$ at $A$ is a principal sub-matrix of $\mathcal{E}(f,B)$. In particular, if $\mathcal{E}(f,B)$ is positive def\/inite, so is~$A$. These comments justify the following lemma.

\begin{Lemma}\label{augmentation}
Let $f$ be the generating function of an isotropic and positive definite kernel on~\mbox{$S^1 \times S^m$} and consider the alternative series representation~\eqref{alter} for~$f$. The following assertions are equivalent:
\begin{enumerate}\itemsep=0pt
\item[$(i)$] the kernel $((x,z), (y,w)) \in (S^1 \times S^m)^2 \to f(x \cdot y, z \cdot w)$ is strictly positive definite;
\item[$(ii)$] If $B$ is an enhanced subset of $S^1 \times S^m$, then the matrix $\mathcal{E}(f,B)$ is positive definite.
\end{enumerate}
\end{Lemma}

Due to the decomposition for the generating function $f$ of an isotropic positive def\/inite kernel as described in~(\ref{alter}), we can write
\begin{gather*}
\mathcal{E}(f,B)=\sum_{l=0}^\infty \mathcal{E}(f,B,l)
\end{gather*}
in which $\mathcal{E}(f,B,l)$ is the interpolation matrix of the kernel $(t,s)\in S^1 \times S^m \to f_l(t)P_{l}^{m}(s)$ at~$B$. The order in which the elements of an enhanced subset of $S^1 \times S^m$ appears forces the matrix $\mathcal{E}(f,B,l)$ to have a very distinctive block representation. Precisely,
\begin{gather*}\mathcal{E}(f,B,l)=\left( \begin{matrix} M_{11} & M_{12} \\ M_{21} & M_{22} \end{matrix} \right),\end{gather*}
where each block $M_{\rho\sigma}=M_{\rho\sigma}(f,B,l)$ has its own block structure
\begin{gather*}M_{\rho\sigma}=[M_{\rho\sigma}^{\mu\nu}]_{\mu,\nu=1}^q,\qquad \rho, \sigma=1,2,\end{gather*}
def\/ined by
\begin{gather*}M_{\rho\sigma}^{\mu\nu}=[f_l(x_i \cdot x_j) ]_{i,j=1}^p(-1)^{l(\sigma+\rho)}P_l^m(z_\mu \cdot z_\nu), \qquad \mu,\nu=1,2,\ldots, q.\end{gather*}
Implicitly used in the writing of the block decomposition above is the fact that Gegenbauer polynomials of even degree are even functions while those of odd degree are odd functions. In particular, since $M_{22}=M_{11}$ and $M_{12}=M_{21}=(-1)^lM_{11}$, the matrix $\mathcal{E}(f,B,l)$ depends upon~$M_{11}$ only.

Keeping all the notation introduced so far, Lemmas~\ref{basic} and~\ref{augmentation} and the comments above lead to the following characterization for the strict positive def\/initeness of the kernel $((x,z), (y,w)) \in (S^1 \times S^m)^2 \to f(x \cdot y, z \cdot w)$ via
\begin{gather*}M_{11}=\big[ [f_l(x_i \cdot x_j) ]_{i,j=1}^p P_l^m(z_\mu \cdot z_\nu)\big]_{\mu,\nu=1}^q.\end{gather*}
\begin{Lemma}
Let $f$ be the generating function of an isotropic and positive definite kernel on~\mbox{$S^1 \times S^m$} and consider the alternative series representation~\eqref{alter} for~$f$. The following assertions are equivalent:
\begin{enumerate}\itemsep=0pt
\item[$(i)$] the kernel $((x,z), (y,w)) \in (S^1 \times S^m)^2 \to f(x \cdot y, z \cdot w)$ is strictly positive definite;
\item[$(ii)$] if $B$ is an enhanced subset of $S^1 \times S^m$ generated by a subset $\{x_1,x_2, \ldots, x_p\}$ of~$S^1$ and the antipodal free subset $\{z_1, z_2,\ldots, z_q\}$ of~$S^m$, then the only solution $(c_1, c_2) \in (\mathbb{R}^{pq})^2$ of the system
\begin{gather*}\big[c_1+(-1)^l c_2\big]^tM_{11}\big[c_1+(-1)^lc_2\big]=0,\qquad l\in \cup_{k}J_f^k,\end{gather*}
is the trivial one, that is, $c_1=c_2=0$.
\end{enumerate}
\end{Lemma}

Introducing components for the vectors $c_i$ in the previous lemma, we obtain the following reformulation.

\begin{Lemma} \label{enhanced}
Let $f$ be as in the previous lemma. The following assertions are equivalent:
\begin{enumerate}\itemsep=0pt
\item[$(i)$] the kernel $((x,z), (y,w)) \in (S^1 \times S^m)^2 \to f(x \cdot y, z \cdot w)$ is strictly positive definite;
\item[$(ii)$] if $B$ is an enhanced subset of $S^1 \times S^m$ generated by a subset $\{x_1,x_2, \ldots, x_p\}$ of $S^1$ and the antipodal free subset $\{z_1, z_2,\ldots, z_q\}$ of $S^m$, then the only solution $(c_1^1, c_1^2, \ldots, c_1^q,c_2^1, c_2^2,$ $ \ldots, c_2^q) \in (\mathbb{R}^{p})^{2q}$ of the system
\begin{gather*}\sum_{\mu,\nu=1}^q \big\{(c_1^\mu+(-1)^l c_2^\mu)^t [f_l(x_i \cdot x_j) ]_{i,j=1}^p (c_1^\nu +(-1)^l c_2^\nu)\big\} P_l^m(z_\mu \cdot z_\nu)=0,\qquad l\in \cup_{k}J_f^k,
\end{gather*}
is the trivial one.
\end{enumerate}
\end{Lemma}

Next, we break up the system in the previous lemma, according to the parity of the elements in $\cup_{k}J_f^k$.

\begin{Proposition} \label{evenodd}
Let $f$ be as in the previous lemma. The following assertions are equivalent:
\begin{enumerate}\itemsep=0pt
\item[$(i)$] the kernel $((x,z), (y,w)) \in (S^1 \times S^m)^2 \to f(x \cdot y, z \cdot w)$ is strictly positive definite;
\item[$(ii)$] if $p$ and $q$ are positive integers, $x_1,x_2, \ldots, x_p$ are distinct points on $S^1$ and $\{z_1, z_2, \ldots,$ $z_q\}$ is an antipodal free subset of $S^m$, then the only solution $(d_1^1, d_1^2,\ldots, d_1^q, d_2^1,d_2^2, \ldots, d_2^q)$ in~$(\mathbb{R}^{p})^{2q}$ of the system
\begin{gather*}
\sum_{\mu,\nu=1}^q \big\{(d_1^\mu)^t [f_l(x_i \cdot x_j)]_{i,j=1}^p d_1^\nu\big\} P_l^m(z_\mu \cdot z_\nu )=0, \qquad l\in (2\mathbb{Z}_++1) \cap \big(\cup_{k}J_f^k\big),\\
\sum_{\mu,\nu=1}^q \big\{(d_2^\mu)^t [f_l(x_i \cdot x_j)]_{i,j=1}^p d_2^\nu\big\} P_l^m(z_\mu \cdot z_\nu )=0, \qquad l\in (2\mathbb{Z}_+) \cap \big(\cup_{k}J_f^k\big),
\end{gather*}
is the trivial one.
\end{enumerate}
\end{Proposition}
\begin{proof}
Assume that for some distinct points $x_1, x_2, \ldots, x_p$ in $S^1$ and some antipodal free subset $\{z_1, z_2, \ldots, z_q\}$ of~$S^m$, the system in $(ii)$ has a~nontrivial solution $(d_1^1, d_1^2,\ldots, d_1^q, d_2^1,$ $d_2^2, \ldots, d_2^q)$. If $d_1^\mu\neq 0$ for some $\mu$, we def\/ine $c_1^\mu=-c_2^\mu=2^{-1}d_1^\mu$, $\mu=1,2,\ldots, q$. Otherwise, we def\/ine $c_1^\mu=c_2^\mu=2^{-1}d_2^\mu$, $\mu=1,2,\ldots, q$. In both cases, the vector $(c_1^1, c_1^2, \ldots, c_1^q,c_2^1, c_2^2, \ldots, c_2^q)$ is nonzero and, in addition,
\begin{gather*}\sum_{\mu,\nu=1}^q \big\{\big(c_1^\mu+(-1)^l c_2^\mu\big)^t [f_l(x_i \cdot x_j) ]_{i,j=1}^p \big(c_1^\nu +(-1)^l c_2^\nu\big)\big\} P_l^m(z_\mu \cdot z_\nu)=0,\end{gather*}
for $l \in [(2\mathbb{Z}_++1) \cap (\cup_{k}J_f^k)] \cup [(2\mathbb{Z}_+) \cap (\cup_{k}J_f^k)]= \cup_{k}J_F^k$. In other words, Condition $(ii)$ in Lemma~\ref{enhanced} does not hold for the enhanced set $B$ generated by the subset $\{x_1,x_2, \ldots, x_p\}$ of $S^1$ and the antipodal free subset $\{z_1, z_2,\ldots, z_q\}$ of $S^m$. Thus, $((x,z), (y,w)) \in (S^1 \times S^m)^2 \to f(x \cdot y, z \cdot w)$ is not strictly positive def\/inite. Conversely, if $(i)$ does not hold, the pre\-vious lemma assures the existence of a subset $\{x_1,x_2, \ldots, x_p\}$ of~$S^1$, an antipodal free subset $\{z_1, z_2,\ldots, z_q\}$ of $S^m$, an enhanced subset $A$ of $S^1 \times S^m$ generated by them and a nonzero vector $(c_1^1, c_1^2, \ldots, c_1^q,c_2^1, c_2^2, \ldots, c_2^q) \in (\mathbb{R}^{p})^{2q}$ so that
\begin{gather*}
\sum_{\mu,\nu=1}^q \big\{\big(c_1^\mu+(-1)^l c_2^\mu\big)^t [f_l(x_i \cdot x_j) ]_{i,j=1}^p \big(c_1^\nu +(-1)^l c_2^\nu\big)\big\} P_l^m(z_\mu \cdot z_\nu)=0,\qquad l\in \cup_{k}J_f^k.
\end{gather*}
However, this last piece of information corresponds to
\begin{gather*}
\sum_{\mu,\nu=1}^q \big\{\big(c_1^\mu -c_2^\mu\big)^t [f_l(x_i \cdot x_j) ]_{i,j=1}^p \big(c_1^\nu -c_2^\nu\big)\big\} P_l^m(z_\mu \cdot z_\nu )=0, \qquad l\in (2\mathbb{Z}_++1) \cap \big(\cup_{k}J_f^k\big),
\end{gather*}
and
\begin{gather*}
\sum_{\mu,\nu=1}^q \big\{\big(c_1^\mu +c_2^\mu\big)^t [f_l(x_i \cdot x_j)]_{i,j=1}^p\big(c_1^\nu +c_2^\nu\big)\big\} P_l^m(z_\mu \cdot z_\nu )=0, \qquad l\in (2\mathbb{Z}_+) \cap \big(\cup_{k}J_f^k\big).
\end{gather*}
On the other hand, it is easily verif\/iable that the vector
\begin{gather*}\big(c_1^1-c_2^1, c_1^2-c_2^2,\ldots, c_1^q-c_2^q,c_1^1+c_2^1, c_1^2+c_2^2,\ldots, c_1^q+c_2^q\big) \in \big(\mathbb{R}^{p}\big)^{2q}\end{gather*}
is nonzero. Thus, $(ii)$ does not hold due to Lemma~\ref{enhanced}.
\end{proof}

We are about ready to prove the crucial result in this section.

\begin{Theorem} \label{key}
Let $f$ be the generating function of an isotropic and positive definite kernel on~\mbox{$S^1 \times S^m$} and consider the alternative series representation~\eqref{alter} for $f$. The following assertions are equivalent:
\begin{enumerate}\itemsep=0pt
\item[$(i)$] the kernel $((x,z), (y,w)) \in (S^1 \times S^m)^2 \to f(x \cdot y, z \cdot w)$ is strictly positive definite;
\item[$(ii)$] if $p$ is a positive integer, $x_1, x_2, \ldots, x_p$ are distinct points in $S^1$ and $c\in \mathbb{R}^p\setminus\{0\}$, then the set
\begin{gather*}\big\{l\in \cup_{k}J_f^k \colon c^{t} [f_l(x_i \cdot x_j)]_{i,j=1}^p c>0\big\}\end{gather*}
contains infinitely many even and infinitely many odd integers.
\end{enumerate}
\end{Theorem}
\begin{proof} Lemma~\ref{prelim} justif\/ies one implication. As for the other, assume the condition in the statement of the theorem holds but
$((x,z), (y,w)) \in (S^1 \times S^m)^2 \to f(x \cdot y, z \cdot w)$ is not strictly positive def\/inite. Hence, we can f\/ind a positive integer $p$, distinct points $x_1, x_2, \ldots, x_p$ in $S^1$, an antipodal free subset $\{y_1, y_2, \ldots, y_n\}$ of $S^m$ and a~nonzero vector $(d_1^1, d_1^2,\ldots, d_1^q, d_2^1, d_2^2, \ldots, d_2^q)$ in $(\mathbb{R}^{p})^{2q}$ so that the two equations in Proposition~\ref{evenodd}(ii) hold. We will proceed assuming that $(d_2^1, d_2^2, \ldots, d_2^q)$ is a nonzero vector and that
\begin{gather*}\sum_{\mu,\nu=1}^q \big\{(d_2^\mu)^t [f_l(x_i \cdot x_j)]_{i,j=1}^p d_2^\nu \big\}P_l^m(z_\mu \cdot z_\nu )=0, \qquad l\in 2\mathbb{Z}_+ \cap \big(\cup_{k}J_f^k\big)\end{gather*}
and will reach a contradiction. The other possibility can be handled similarly but the details will be omitted. Without loss of generality, we can assume that the vector $d_2^1$ is nonzero. Since the set
\begin{gather*}\big\{l\in \cup_{k}J_f^k \colon \big(d_2^1\big)^t [f_l(x_i \cdot x_j)]_{i,j=1}^p d_2^1>0\big\}\cap 2\mathbb{Z} \end{gather*}
is inf\/inite by assumption, we can select an inf\/inite subset $Q$ of it and a number $\theta=\theta(l)$ in $\{1,2,\ldots, q\}$ so that
\begin{gather*}\big(d_2^\theta\big)^t [f_l(x_i \cdot x_j)]_{i,j=1}^p d_2^\theta \geq \big(d_2^\mu\big)^t [f_l(x_i \cdot x_j)]_{i,j=1}^p d_2^\mu, \qquad \mu=1,2,\ldots,q,\qquad l \in Q.\end{gather*}
In particular,
\begin{gather*}\big(d_2^\theta\big)^t [f_l(x_i \cdot x_j)]_{i,j=1}^p d_2^\theta >0, \qquad l \in Q.\end{gather*}
Returning to the initial equality we can write
\begin{gather*}
0 = 1+\sum_{\substack{\mu=1 \\ \mu \neq \theta}}^q \frac{(d_2^\mu)^t [f_l(x_i \cdot x_j)]_{i,j=1}^p d_2^\mu}{(d_2^\theta)^t [f_l(x_i \cdot x_j)]_{i,j=1}^pd_2^\theta}\frac{P_l^m(z_\mu\cdot z_\mu)}{P_l^m(1)}\\
\hphantom{0 =1}{}+\sum_{\mu\neq \nu} \frac{(d_2^\mu)^t [f_l(x_i \cdot x_j)]_{i,j=1}^p d_2^\nu}{(d_2^\theta)^t [f_l(x_i \cdot x_j)]_{i,j=1}^p d_2^\theta}\frac{P_l^m(z_\mu\cdot z_\nu)}{P_l^m(1)},\qquad l \in Q.
\end{gather*}
Since each $f_l$ is the continuous and isotropic part of a positive def\/inite kernel on $S^1$, we have that $(d_2^\mu)^t [f_l(x_i \cdot x_j)]_{i,j=1}^p d_2^\mu\geq 0$, $\mu=1,2,\ldots, q$. In particular,
\begin{gather*}
\sum_{\substack{\mu=1 \\ \mu \neq \theta}}^q \frac{(d_2^\mu)^t [f_l(x_i \cdot x_j)]_{i,j=1}^p d_2^\mu}{(d_2^\theta)^t [f_l(x_i \cdot x_j)]_{i,j=1}^p d_2^\theta} \in [0,q-1],
\end{gather*}
while the Cauchy--Schwarz inequality implies that
\begin{gather*}
0\leq \left|\frac{(d_2^\mu)^t [f_l(x_i \cdot x_j)]_{i,j=1}^p d_2^\nu}{(d_2^\theta)^t [f_l(x_i \cdot x_j)]_{i,j=1}^p d_2^\theta}\right|\leq 1, \qquad \mu \neq \nu.
\end{gather*}
Since $z_\mu \cdot z_\nu \in (-1,1)$, $\mu \neq \nu$, a well-known property of the Gegenbauer polynomials provides the limit formula~\cite[p.~196]{szego}
\begin{gather*}\lim_{\substack{l\to \infty \\ l\in Q}} \frac{P_l^m(z_\mu\cdot z_\nu)}{P_l^m(1)}=0, \qquad \mu \neq \nu.\end{gather*}
Consequently, we may apply the def\/inition of limit conveniently ($l$ large enough), to conclude that $0\geq 1+0-1/2=1/2$, a contradiction.
\end{proof}

The next theorem demands the truncated sum functions ($\gamma\geq 0$)
\begin{gather*}f_\gamma^{\rm o}=\sum_{2l+1\geq \gamma} f_{2l+1} \qquad \mbox{and} \qquad f_\gamma^{\rm e}= \sum_{2l \geq \gamma} f_{2l}\end{gather*}
attached to the generating function of an isotropic and positive def\/inite kernel. Since $P_l^m(1)\geq 1$, $m\geq 2$, $l\geq 0$, it follows that
\begin{gather*}
|f_l(t)| \leq \sum_{k=0}^\infty a_{k,l}^{1,m}P_k^1(1)\leq \sum_{k=0}^\infty a_{k,l}^{1,m}P_k^1(1)P_l^m(1),\qquad l\geq \gamma.
\end{gather*}
In particular, since $\sum\limits_{l=0}^\infty \sum\limits_{k=0}^\infty a_{k,l}^{1,m}P_k^1(1)P_l^m(1) <\infty$, the functions $f_\gamma^{\rm o}$ and $f_\gamma^{\rm e}$ are continuous.

\begin{Theorem}
Let $f$ be the generating function of an isotropic and positive definite kernel on~\mbox{$S^1 \times S^m$} and consider the alternative series representation~\eqref{alter} for~$f$. The following assertions are equivalent:
\begin{enumerate}\itemsep=0pt
\item[$(i)$] the kernel $((x,z), (y,w)) \in (S^1 \times S^m)^2 \to f(x \cdot y, z \cdot w)$ is strictly positive definite;
\item[$(ii)$] for each $\gamma \geq 0$, the functions~$f_\gamma^{\rm o}$ and~$f_{\gamma}^{\rm e}$ are the generating functions of isotropic and strictly positive definite kernels on~$S^1$.
\end{enumerate}
\end{Theorem}
\begin{proof} If $(i)$ holds, we can apply Lemma \ref{prelim} in order to see that
\begin{gather*}c^t f_\gamma^{\rm o}(x_i \cdot x_j)c =\sum_{2l+1\geq \gamma} c^t f_{2l+1}(x_i \cdot x_j)c >0,\end{gather*}
whenever $c \in \mathbb{R}^p\setminus\{0\}$, $\gamma \geq 0$ and $x_1, x_2, \ldots, x_p$ are distinct points in~$S^1$. Obviously, a similar property is valid for
$f_\gamma^{\rm e}$. Thus, $(ii)$ follows. Conversely, if $(ii)$ holds, then Condition $(ii)$ in the previous theorem holds as well, due to the fact that $(ii)$ is valid for all $\gamma\geq 0$. Thus, Theo\-rem~\ref{key} guarantees the strict positive def\/initeness of $((x,z), (y,w)) \in (S^1 \times S^m)^2 \to f(x \cdot y, z \cdot w)$.
\end{proof}

The coef\/f\/icient in the series expansion of $f_\gamma^{\rm o}$ attached to the polynomial $P_k^1$ is
\begin{gather*}\sum_{\gamma \leq 2l+1 \in J_f^k} a_{k,2l+1}^{1,m}.\end{gather*}
It is positive if, and only if, $J_f^k$ contains an odd integer greater than or equal to $\gamma$. A similar remark applies to the coef\/f\/icients in the series of $f_\gamma^{\rm e}$. Taking into account the characterization for strict positive def\/initeness on~$S^1$ quoted at the introduction, we have the following consequence of the previous theorem and our f\/inal characterization for strict positive def\/initeness on $S^1 \times S^m$.

\begin{Theorem}
Let $f$ be the generating function of an isotropic and positive definite kernel on~\mbox{$S^1 \times S^m$}. The following assertions are equivalent:
\begin{enumerate}\itemsep=0pt
\item[$(i)$] the kernel $((x,z), (y,w)) \in (S^1 \times S^m)^2 \to f(x \cdot y, z \cdot w)$ is strictly positive definite;
\item[$(ii)$] for each $\gamma \geq 0$, the sets
\begin{gather*}\big\{k\in \mathbb{Z}\colon J_f^{|k|} \cap \{\gamma,\gamma+1,\ldots\}\cap (2\mathbb{Z}+1)\neq \varnothing\big\}\end{gather*}
and
\begin{gather*}\big\{k\in \mathbb{Z}\colon J_f^{|k|} \cap \{\gamma,\gamma+1,\ldots\}\cap 2\mathbb{Z}_+\neq \varnothing\big\} \end{gather*}
intersect every arithmetic progression in $\mathbb{Z}$.
\end{enumerate}
\end{Theorem}

\section[Replacing $S^m$ with a compact two-point homogeneous space]{Replacing $\boldsymbol{S^m}$ with a compact two-point homogeneous space}\label{section4}

The results obtained so far in the paper can be adapted to hold for kernels on a product of the form $S^1 \times \mathbb{M}^d$, in which $\mathbb{M}^d$ is a compact two-point homogeneous space. The case $S^1 \times S^{1}$ was covered in~\cite{jean1} while the case~$S^d\times\mathbb{M}^d$, $d\geq 3$, was covered in \cite[Theorem~4.5]{barbosa1}. The results sketched in this section complement these two cases.

Let us write $|zw|$ to denote the usual normalized surface distance between~$z$ and~$w$ in~$\mathbb{M}^d$. As described in~\cite{barbosa1}, an isotropic kernel $((x,z), (y,w)) \in (S^1 \times \mathbb{M}^d)^2 \to f(x \cdot y, \cos (|zw|/2))$ is positive def\/inite if, and only if, the generating function $f$ has a double series representation in the form
\begin{gather*}
f(t,s) = \sum_{k,l=0}^\infty a_{k,l}^dP_k^{1}(t)P_l^{d,\beta}(s), \qquad t,s \in [-1,1]^2,
\end{gather*}
in which $a_{k,l}^d\geq 0$, $k,l\in \mathbb{Z}_+$, $P_l^{d,\beta}$ is the Jacobi polynomial associated to the pair $((d-2)/2,\beta)$, $\beta$ is a number from the list $-1/2, 0, 1, 3$, depending on the respective category $\mathbb{M}^d$ belongs to, that is, the real projective spaces $\mathbb{P}^d(\mathbb{R})$, $d=2,3,\ldots$, the complex projective spaces $\mathbb{P}^d(\mathbb{C})$, $d=4,6,\ldots$, the quaternionic projective spaces $\mathbb{P}^d(\mathbb{H})$, $d=8,12,\ldots$, and the Cayley projective plane $\mathbb{P}^{d}({\rm Cay})$, $d=16$, and $\sum\limits_{k,l=0}^\infty a_{k,l}P_k^{1}(1)P_l^{d,\beta}(1)<\infty$.

In this setting, the procedure adopted in Section~\ref{section3} can be considerably simplif\/ied. An alternative series representation for the generating function of the kernel can be likewise def\/ined and the fact that a point in $\mathbb{M}^d$ possesses inf\/initely many antipodal points permits the deduction of a version of Theorem \ref{key} without considering any enhancements and augmentations. Precisely, we have the following result.

\begin{Theorem} \label{key1}
Let $f$ be the generating function of an isotropic and positive definite kernel on~\mbox{$S^1 \times \mathbb{M}^d$} and consider the alternative series representation for~$f$. The following assertions are equivalent:
\begin{enumerate}\itemsep=0pt
\item[$(i)$] the kernel $((x,z), (y,w)) \in (S^1 \times \mathbb{M}^d)^2 \to f(x \cdot y, z \cdot w)$ is strictly positive definite;
\item[$(ii)$] if $p$ is a positive integer, $x_1, x_2, \ldots, x_p$ are distinct points in $S^1$ and $c\in \mathbb{R}^p\setminus\{0\}$, then the set
\begin{gather*}
\big\{l\in \cup_{k}J_f^k \colon c^{t} [f_l(x_i \cdot x_j) ]_{i,j=1}^p c>0\big\}
\end{gather*}
contains infinitely many integers.
\end{enumerate}
\end{Theorem}

Taking into account the characterization for strict positive def\/initeness obtained in \cite{barbosa}, the f\/inal characterization in these remaining cases is this one.

\begin{Theorem}
Let $f$ be the generating function of an isotropic and positive definite kernel on~\mbox{$S^1 \times \mathbb{M}^d$}. Assume $\mathbb{M}^d$ is not a~sphere. The following assertions are equivalent:
\begin{enumerate}\itemsep=0pt
\item[$(i)$] the kernel $((x,z), (y,w)) \in (S^1 \times \mathbb{M}^d)^2 \to f(x \cdot y, z \cdot w)$ is strictly positive definite;
\item[$(ii)$] for each $\gamma \geq 0$, the set
\begin{gather*}
\big\{k\in \mathbb{Z}\colon \ J_f^{|k|} \cap \{\gamma,\gamma+1,\ldots\}\neq \varnothing\big\}
\end{gather*}
intersects every arithmetic progression in $\mathbb{Z}$.
\end{enumerate}
\end{Theorem}

\subsection*{Acknowledgements}

The authors are grateful to the referees for their valuable comments and suggestions. Second author acknowledges partial f\/inancial support from FAPESP, grant 2014/00277-5. Likewise, the third author acknowledges support from the same foundation, under grants 2014/25796-5 and 2016/03015-7.

\pdfbookmark[1]{References}{ref}
\LastPageEnding

\end{document}